\documentclass[12 pt]{amsart}
\usepackage{amsthm, amsfonts, amssymb, color}
\usepackage{mathrsfs}
\usepackage{amsmath}
\usepackage{amstext, amsxtra}
\usepackage{txfonts}
\usepackage[colorlinks, linkcolor=black, citecolor=blue, pagebackref, hypertexnames=false]{hyperref}
\usepackage{comment}

\allowdisplaybreaks
\usepackage{geometry}
\newtheorem{theorem}{Theorem}[section]

\newtheorem{corollary}[theorem]{Corollary}
\newtheorem{lemma}[theorem]{Lemma}

\newtheorem{assumption}{Assumption}[section]
\newtheorem{remark}[theorem]{Remark}

\newcommand{\s}{ L}
\newcommand{\eq}{\begin{equation}}
\newcommand{\eeq}{\end{equation}}

\numberwithin{equation}{section}

\title   [] { $\s^p$ Boundedness of the Scattering Wave Operators of Schr\"odinger Dynamics  -Part \uppercase\expandafter{\romannumeral2} }

\author{   Avy Soffer }

\address{Department of Mathematics\\
Rutgers University\\
110 Frelinghuysen Rd.\\
Piscataway, NJ, 08854, USA}

\email{soffer@math.rutgers.edu}

\author{Xiaoxu Wu}
\address{Department of Mathematics\\
Rutgers University\\
110 Frelinghuysen Rd.\\
Piscataway, NJ, 08854, USA}
\email{xw292@math.rutgers.edu}

\thanks{
A.Soffer is supported in part by Simons Foundation Grant number 851844
}

\begin{document}

\begin{abstract}
We give another proof of the $L^p$ boundedness of scattering wave operators, at the low frequency part of the data.
The proof also allows the control of the commutator of multiplication by $|x|$ with the wave operator in $L^p.$
The method we develop here is geared to proving similar results for \emph{time dependent potentials},
complimenting previous work focused on the high frequency part \cite{SW2020}.\end {abstract}
\maketitle
\textbf{Introduction}\noindent

The dispersive estimates for solutions of Schr\"odinger type equations play a central role in both linear and nonlinear scattering, inverse scattering and existence theory.
One way of extending dispersive estimates to complicated systems, is via intertwining the full dynamics with a simpler one. For this we use the scattering wave operators. Therefore, the $L^p$ mapping properties of the wave operator, and its commutators with derivatives and $|x|$ are crucial.

 In this note we provide a proof of $\s^p$ boundedness of the wave operator $\Omega_\pm$ on the low frequency cut-off $\s^p$ space $\beta(|P|\leq M)\s^p_x,$ for time-independent Schr\"odinger equation:
\eq
i\partial_t\psi(x,t)=H_0\psi(x,t)+V(x)\psi(x,t), \quad (x,t)\in \mathbb{R}^3\times\mathbb{R}\label{sc}
\eeq
with $\psi(x,0)=\psi_0\in \s^2_x$, where $H_0:=-\Delta_x$ and $V(x)$ is a real function, and $H=H_0+V$ satisfying assumption \ref{asp1} below.

\begin{assumption}\label{asp1}There exists some $\delta>0$ such that for $z=a\pm ib, a,b\in \mathbb{R}, b\geq 0$,
\eq
\| P_c\frac{d^j}{dz^j}[R(z)]\|_{\mathcal{L}_{\delta, x}^2\to\mathcal{L}_{-\delta, x}^2 }\lesssim_M O( |z|^{\frac{1}{2}-j}), j=1,2,3,4,\quad |z|\leq 4M,\label{RVd}
\eeq
and
\eq
\| P_c R(z)\|_{\mathcal{L}_{\delta, x}^2\to\mathcal{L}_{-\delta, x}^2 }\lesssim O( 1),\quad |z|\leq 4M.\label{RV0}
\eeq
\end{assumption}
\begin{remark}\label{rem1}
\eq
\| \frac{d^j}{dz^j}[R(z)]P_c\|_{\mathcal{L}_{\delta, x}^2\to\mathcal{L}_{-\delta, x}^2 }\lesssim_M O( |z|^{\frac{1}{2}-j}), \quad |z|\leq 4M,\label{RV00}
\eeq
is valid for $\delta> \frac{1}{2}+j, j=0, 1,2,3,4,$ see \cite{KK2014}.
\end{remark}
Here
\eq
\mathcal{K}_{\delta,x}^1(\mathbb{R}^3):=\left\{V(x): \sup\limits_{|x|\leq 4} \int d^3k \frac{\langle x-k\rangle^\delta|V(x-k)|}{|k|}<\infty \right\},
\eeq
\eq
\mathcal{L}^p_{\delta,x}(\mathbb{R}^3):=\left\{ f(x): \langle x\rangle^\delta f(x)\in \s^p_x(\mathbb{R}^3)  \right\}, \quad 1\leq p\leq\infty
\eeq
with $\langle \cdot\rangle: \mathbb{R}^n \to \mathbb{R}, x\mapsto \sqrt{|x|^2+1}$. We write $\mathcal{K}_{0,x}^1$ as $\mathcal{K}$ for simplicity. Recall that in \cite{SW2020}, we proved $\s^p$ boundedness of wave operator on high frequency cut-off $\s^p$ space for a class of time-dependent potentials including time-independent cases.

For the low fequency part, one needs different methods. Here we develop another way of proving the $L^p-$boundedness of wave operators at low frequencies. This approach is suitbale for dealing with time dependent potentials. It is based on a representation of the integral equation for the wave operator in terms of oscillatory integrals, as an alternative to the standard resolvents. We will not insist on optimal conditions on the potentials, as the extension to time dependnet potentials would require extra assumptions.
We will also show that the analysis here allows the estimate in $L^p$ of the commutator $$
[|x|,\Omega].$$
 
\subsection{Main result and outline of its proof}
Let $R_0(z)$ denote the free resolvent for $z=a\pm ib, a,b\in \mathbb{R}, b\geq 0$, $R(z)$, the perturbed one and for $k\in \mathbb{R}$,
\eq
R_0^\pm(k):=\lim\limits_{\epsilon \downarrow 0} R_0(k\pm i\epsilon),
\eeq
\eq
R^\pm(k):=\lim\limits_{\epsilon \downarrow 0} R(k\pm i\epsilon).
\eeq
Let $P_c$ denote the projection on the continuous spectrum of $H:=H_0+V(x)$.

Here is the main result in this paper:
\begin{theorem}\label{thm1}If there exists some $\delta>0$ such that $V(x)$ satisfies assumption \ref{asp1} and if $V(x)\in \s^1_{\delta,x}\cap \mathcal{K}\cap \s^\infty_{2\delta,x}$ for the same $\delta$, then $\Omega_+\beta(|P|\leq M): \mathcal{L}^p_x\to \mathcal{L}^p_x$, and  $\beta(|P|\leq M)\Omega_+^*: \mathcal{L}^p_x\to \mathcal{L}^p_x$, are bounded, $1\leq p\leq\infty$.
\end{theorem}
\begin{remark}Based on Remark \ref{rem1}, $\delta>9/2$ is sufficient. $\delta>9/2$ can be improved to $\delta >5/2$ if we only have to estimate wave operator on $\s^p_x$ space since $2+\epsilon$ many derivatives on $P_cR^-(z)$ will be sufficient. Here we have to take $4$ many derivatives in order to estimate $[|x|, \Omega_\pm \beta(|P|\leq M)]$ on $\s^p_x$. 
\end{remark}
\begin{remark}The proof works as well if we have that for $z=a+ib$, $a,b\in \mathbb{R}, b\geq 0$,
\eq
\| R(z)P_c\|_{\mathcal{B}_1\to\mathcal{B}_2}\sim O( 1),\quad |z|\leq 4M\label{June26-1}
\eeq
and
\eq
\| \frac{d^j}{dz^j}[R(z)]P_c\|_{\mathcal{B}_1\to\mathcal{B}_2 }\sim_M O( |z|^{\frac{1}{2}-j}), \quad j=1,2,\quad |z|\leq 4M\label{June26-2}
\eeq
for some Banach space $\mathcal{B}_1,\mathcal{B}_2$ instead, see \cite{G2004}, \cite{GS2004}, \cite{G2006} and the reference therein. Resolvent estimate is closely related to dispersive estimate. For a survey of dispersive estimate, see \cite{S2007}. For small time-dependent perturbation, see \cite{RS2004}. 
\end{remark}
\begin{remark}This note is focused on the case when $V(x)$ is generic. In other word, $0$ is neither an eigenvalue of $H$ nor a threshold. Recall that a resonance is a distributional
solution of $H\psi=0$ so that $\psi\notin \s^2_{x}$ but $\langle  x\rangle^{-\delta}\psi\in \s^2_x$, for any $\delta>\frac{1}{2}$ see \cite{JK1979}. When $V(x)$ is generic, it is known that all the bounded states of $H$ are localized in $x$ in the sense that $P_c: \s^2_{\delta,x}\to \s^2_{\delta,x}$, are bounded for all $\delta\in \mathbb{R}$, see \cite{KK2014} and the reference therein.
\end{remark}
The same result can be extended to any other higher dimensions.\par
\textbf{Outline of the proof: }
Above all, we give an representation for $\Omega_+$(The one for $\Omega_-$ is similar.):
\begin{lemma}[Representation for $\Omega_+$]\label{ReWave}In $3$ dimensions,
\eq
\Omega_+\psi(x)=P_c\psi(x) -\frac{1}{(2\pi)^3}\int d^3 q d^3y  P_cR^-(q^2)V(x)e^{i(x-y)\cdot q} \psi(y).\label{June26-re}
\eeq 
\end{lemma}
Let
\eq
I\psi(x):=-\frac{1}{(2\pi)^3}\int d^3 q d^3y  P_cR^-(q^2)V(x)e^{i(x-y)\cdot q} \psi(y).
\eeq
\begin{lemma}\label{sphere}In $\mathbb{R}^3$,
\eq
\int_{S^2}d\sigma(q)e^{ix\cdot q}=4\pi i\sin(|x||q|).
\eeq
\end{lemma}
Based on the second resolvent identity, the kernel of the free resolvent $R_0^-(q^2)$ 
\eq
R_0^-(q^2)f(x)=c\int d^3k \frac{e^{-iq|x-y|}}{ |x-y|}f(y),\quad \text{for }q\geq 0,\text{ some }c>0\label{kernel}
\eeq
and some elementary integral of $\hat{q}$ over sphere $S^2$(Lemma \ref{sphere}), we derive the following representation for $I\psi(x)$
\begin{multline}
I\psi(x)= -\frac{ic}{(2\pi)^3}\int_{[0,\infty)\times \mathbb{R}^6} |q|d |q| d^3kd^3y  \frac{e^{-i|x-k||q|}}{|x-k|}P_cV(k)\frac{\sin (|k-y||q|)}{|k-y|} \psi(y)+\\
\frac{ic}{(2\pi)^3}\int_{[0,\infty)\times \mathbb{R}^9} |q|d |q| d^3kd^3yd^3z  \frac{e^{-i|x-k||q|}}{|x-k|}V(k)R^-(q^2; z)P_cV(k-z)\frac{\sin (|k-z-y||q|)}{|k-y|} \psi(y)\\
=:I_1\psi(x)+I_2\psi(x)\label{I1I2}
\end{multline}
for some $c>0$ where $R^-(q^2; z)$ denote the kernel of $R^-(q^2)$. Using lemma \ref{Osi}, we get $\s^p$ boundedness of $I_1\psi(x)$ and $I_2\psi(x)$:
\begin{lemma}\label{LpI11}If $V(x)\in\mathcal{K}\cap \s^1_{x}$, then
\eq
\|I_2\psi(x)\|_{\s^p_x}\lesssim_M\|\psi(x)\|_{\s^p_x}, \text{ for all }1\leq p\leq\infty.
\eeq
\end{lemma}
\begin{lemma}\label{LpI21}If $R(z) $ satisfies Assumption \ref{asp1}, and if $V(x)\in \s^1_{\delta,x}\cap \mathcal{K}\cap \s^\infty_{2\delta,x}$ for the same $\delta$, then for any $y\in \mathbb{R}^3$,
\eq
\|I_2\psi(x)\|_{\s^p_x}\lesssim_M\|\psi(x)\|_{\s^p_x}.
\eeq
\end{lemma}
\begin{remark}\eqref{June26-re} is equivalent to the stationary representation formula(in terms of resolvents) of the wave operator
\eq
\Omega_+=1-\frac{1}{2\pi i}\int_{-\infty}^\infty R^-(\lambda)VR_0^+(\lambda)d\lambda,\label{Ya-re}
\eeq
see \cite{Y1995}. 
\end{remark}
\begin{remark}
The merit of representation \eqref{June26-re} is to introduce how to get representation by using oscillation integral. And when it comes to time-dependent cases, the oscillation integral will play the same role as the resolvent does in time-independent cases.
\end{remark}
As an application, we prove that for $[|x|, (\Omega_+-P_c)]\beta(|P|\leq M)$ is almost bounded on $\s^p_x$:
\begin{lemma}\label{com}If $V(x)$ satisfies Assumption \ref{asp1} and $V(x)\in \s^1_{\delta+1,x}\cap \mathcal{K}^1_{1,x}\cap \s^\infty_{2\delta,x}$ with $\delta$ defined in Assumption \ref{asp1}, then
\eq
[|x|,(\Omega_+-P_c)\beta(|P|\leq M)]: \s^p_x\to \s^{p+\epsilon}_x+\s^p_x,\label{Op1}
\eeq
is bounded for any $\epsilon>0, p\in [1,\infty)$ and 
\eq
[|x|,(\Omega_+-P_c)\beta(|P|\leq M)]: \s^p_x\to \s^\infty_x,\label{Op2}
\eeq
is bounded for all $p\in [1,\infty)$.
\end{lemma}

\subsection{Other Notation}\label{Notation}
We continue to use some notation from \cite{SW2020}. Let $H:=H_0+V(x).$ $\beta(t)$ is a smooth cut-off function satisfying $\beta(\lambda)=0$ for $-\infty<\lambda<1/2$ and $\beta(\lambda)=1$ for $\lambda\geq 1$ and $\beta(t>M):=\beta(\frac{t}{M})$, $\beta(t\leq M)=1-\beta(t>M)$ for $M>0$. $\chi$ is the usual characteristic function. \par

Throughout this note, we are talking about the case in $3$ dimensions and write $\s^p_x, \s^p_{\delta,x}$ for simplicity.
\section{Time-independent cases}
\subsection{Elementary tools}
Based on the second resolvent identity, write $R^-(q^2)P_c$ as 
\eq
R^-(q^2)P_c=R_0^-(q^2)P_c-R_0^-(q^2)V(x)R^-(q^2)P_c=:R_0^-(q^2)P_c+R_1(q^2).
\eeq
We start with some elementary identity and estimates:
\begin{lemma}\label{resolvent}For $j=1,2,3,4,$ if $V(x)$ satisfies Assumption \ref{asp1} and $V(x)\in \s^\infty_{2\delta,x}$, then for $q\leq 4M$, 
\eq
\| q^{j-1}\frac{d^j}{dq^j}[R_1(q^2)] \|_{\s^1_{\delta,x}\to\s^\infty_{x}}\lesssim_M 1.
\eeq
$\delta$ is defined in Assumption \ref{asp1}.
\end{lemma}
\begin{proof}Due to Assumption \ref{asp1}, it follows by using chain's rule, using second resolvent identity
\eq
R_1(q^2)=-R_0^-(q^2)V(x)R_0^-(q^2)P_c+R_0^-(q^2)V(x)R^-(q^2)P_cV(x)R_0^-(q^2)
\eeq
and using that $R_0^-(q^2): \s^1_{\delta,x}\to \s^2_{-\delta,x}, R_0^-(q^2): \s^2_{\delta,x}\to \s^\infty_{x}, P_c: \s^1_{\delta,x}\to \s^1_{\delta,x} $ are bounded.
\end{proof}
\begin{lemma}\label{main}For $M\geq 2$, let
\eq
I(a,b):=\frac{1}{ab}\int_0^\infty dq q \beta(q\leq M)f(q) \sin(aq)e^{-ibq} \quad\text{for }a,b> 0.
\eeq
If $f\in C^4_q(\mathbb{R})$ satisfying 
\eq
|\partial_{q}^{j+1}[f(q)]|\lesssim \frac{1}{q^{j}}, \text{ for }q\in [0,2M], j=0,1,2,3,
\eeq
with
\eq
C(f):=\sup\limits_{q\in [0,2M]} |  f(q)|+\sup\limits_{q\in [0,2M], j=0,1,2,3} q^j|\partial_{q}^{j+1}[f(q)]|,
\eeq
then
\eq
|I(a,b)|\lesssim \left(\frac{\chi(a+b\geq 1)}{\langle a+b\rangle^2}\times \frac{1}{\langle a-b\rangle^2}(1+\frac{1}{b})+\frac{\chi(a+b<1)}{ab}\right)C(f).\label{I(ab)}
\eeq
\end{lemma}
\begin{proof}Write 
\begin{align}
I(a,b)=&\chi_1(a\in [\frac{b}{2}, 2b])I(a,b)+\chi_1(a>2b)I(a,b)+\chi_1(a<\frac{b}{2})I(a,b)+\chi_0(a+b)I(a,b)\\
=:&I_\sim(a,b)+I_>(a,b)+I_<(a,b)+I_0(a,b),
\end{align}
where 
\eq
\chi_0(k):=\chi(k\geq 1), 
\eeq
\eq
\chi_1(a\in [\frac{b}{2}, 2b]):=(1-\chi_0(a+b))\chi(a\in [\frac{b}{2}, 2b]), \chi_1(a>2b):=(1-\chi_0(a+b))\chi(a>2b) 
\eeq
and
\eq
\chi_1(a>2b):=(1-\chi_0(a+b))\chi(a<b/2).
\eeq
\textbf{Estimate for $I_\sim(a,b)$:} write $\sin(aq)$ as
\eq
\sin(aq)=\frac{1}{2i}(e^{iaq}-e^{-iaq})\label{sin}
\eeq
and plug \eqref{sin} into $I_\sim(a,b)$
\begin{multline}
I_\sim(a,b)=\frac{\chi_1(a\in [\frac{b}{2}, 2b])}{2iab}\int_0^\infty dq q \beta(q\leq M)f(q) e^{i(a-b)q}-\\
\frac{\chi_1(a\in [\frac{b}{2}, 2b])}{2iab}\int_0^\infty dq q \beta(q\leq M)f(q) e^{-i(a+b)q}.
\end{multline}
When $|a-b|\leq 1$, we do nothing while for $|a-b|>1$, we take integration by parts in $q$ variable twice by setting
\eq
e^{i(\pm a-b)q}=\frac{1}{i(\pm a-b)}\partial_q[e^{i(\pm a-b)q}]
\eeq
respectively. In the end, we have
\begin{multline}
|I_\sim(a,b)|\lesssim \frac{\chi_1(a\in [\frac{b}{2}, 2b])}{ab \langle a-b\rangle^2}\times \left(|f(0)| +\| \partial_q^2[q\beta(q\leq M)f(q)]\|_{\s^1_q}\right)\\
(\text{due to }\chi_1(a\in [\frac{b}{2}, 2b]))\lesssim_M   \frac{1}{\langle a+b\rangle^2 \langle a-b\rangle^2}\sup\limits_{q\in [0,2M]} (|f(q)|+|\partial_q[f(q)]|+|q\partial_q^2[f(q)]|)\\
\lesssim_M\frac{1}{\langle a+b\rangle^2 \langle a-b\rangle^2} C(f).\label{sim}
\end{multline}
\textbf{Estimate for $I_{>}(a,b)$:} take integration by parts in $q$ variable by setting 
\eq
\sin(aq)=\frac{1}{a}\partial_{q}[-\cos(aq)]
\eeq
and we have 
\begin{multline}
I_>(a,b)=\frac{\chi(a>2b)}{ab}\times \frac{1}{a}\left( q\beta(q\leq M)f(q)(-\cos(aq))e^{-ibq}\vert_{q=0}^{q=\infty}-\right.\\
\int_0^\infty dq \partial_q[q\beta(q\leq M)f(q) ][-\cos(aq)e^{-ibq}]-\left.\int_0^\infty dq q\beta(q\leq M) f(q)[(ib)\cos(aq)e^{-ibq}]\right)\\
=\frac{\chi(a>2b)}{a^2b}\int_0^\infty dq \partial_q[q\beta(q\leq M)f(q) ]\cos(aq)e^{-ibq}+\frac{i\chi(a>2b)}{a^2}\int_0^\infty dq q\beta(q\leq M)f(q) \cos(aq)e^{-ibq}\\
=:I_{>,1}(a,b)+I_{>,2}(a,b).
\end{multline}
Here we discard the boundary terms due to factor $q$ and factor $\beta(q\leq M)$. 

For $I_{>,2}(a,b)$, we do the same transformation as what we did for $I_\sim(a,b)$ and similarly, we have
\begin{multline}
|I_{>,2}(a,b)|\lesssim \frac{\chi_1(a>2b)}{a^2 \langle a-b\rangle^2}\times \left(|f(0)| +\| \partial_q^2[q\beta(q\leq M)f(q)]\|_{\s^1_q}\right)\\
(\text{due to }\chi_1(a>2b))\lesssim_M   \frac{1}{\langle a+b\rangle^2 \langle a-b\rangle^2}C(f). \label{>2}
\end{multline}
For $I_{>,1}(a,b)$, we take integration by parts in $q$ variable by setting
\eq
\cos(aq)=\frac{1}{a}\partial_q[\sin(aq)]
\eeq
and we have 
\begin{multline}
I_{>,1}(a,b)=\frac{\chi_1(a>2b)}{a^3b} \partial_q[q\beta(q\leq M)f(q) ]\sin(aq)e^{-ibq}\vert_{q=0}^{q=\infty}-\\
\frac{\chi_1(a>2b)}{a^3b}\int_0^\infty dq \partial_q^2[q\beta(q\leq M) f(q)]\sin(aq)e^{-ibq}-\frac{-i\chi_1(a>2b)}{a^3}\int_0^\infty dq \partial_q[q\beta(q\leq M) ]\sin(aq)e^{-ibq}\\
=\frac{-\chi_1(a>2b)}{a^3b}\int_0^\infty dq \partial_q^2[q\beta(q\leq M) f(q)]\sin(aq)e^{-ibq}+\frac{i\chi_1(a>2b)}{a^3}\int_0^\infty dq \partial_q[q\beta(q\leq M) ]\sin(aq)e^{-ibq}\\
=:I_{>,11}(a,b)+I_{>,12}(a,b).
\end{multline}
Here we discard the boundary terms due to factor $\sin(aq)$ and factors $\beta(q\leq M), \partial_q[\beta(q\leq M)]$. 

For $I_{>,12}(a,b)$, we take the same transformation as what we did for $I_\sim(a,b)$ except that here we only take integration by parts once. Similarly, we have
\begin{multline}
|I_{>,12}(a,b)|\lesssim \frac{\chi_1(a>2b)}{a^3 \langle a-b\rangle}\times \left(|f(0)| +\| \partial_q^2[q\beta(q\leq M)f(q)]\|_{\s^1_q}\right)\\
(\text{due to }\chi_1(a>2b))\lesssim_M   \frac{1}{\langle a+b\rangle^2 \langle a-b\rangle^2}C(f).\label{>12}
\end{multline}
For $I_{>,11}(a,b)$, when $q\leq \frac{1}{a}$, we do nothing while for $q>\frac{1}{a}$, we take integration by parts in $q$ variable by setting
\eq
\sin(aq)=\frac{1}{a}\partial_q[-\cos(aq)]
\eeq
twice and we have
\begin{multline}
|I_{>,11}(a,b)|\lesssim \frac{\chi_1(a>2b)}{a^4b}\sup\limits_{q\in [0,2M]}|\partial_q^2[q\beta(q\leq M) f(q)]|+\frac{\chi_1(a>2b)}{a^5b}\sup\limits_{q\in [\frac{1}{a},2M]}|\partial_q^3[q\beta(q\leq M) f(q)]|+\\
\frac{\chi_1(a>2b)}{a^5}\sup\limits_{q\in [\frac{1}{a},2M]}|\partial_q^2[q\beta(q\leq M) f(q)]|+\\
\frac{\chi_1(a>2b)}{a^5b}\left|\int_{1/a}^\infty dq \partial_q^4[q\beta(q\leq M) f(q)]\sin(aq)e^{-ibq} \right|+\frac{\chi_1(a>2b)}{a^5}\left|\int_{1/a}^\infty dq \partial_q^3[q\beta(q\leq M)f(q) ]\sin(aq)e^{-ibq} \right|\\
+\frac{\chi_1(a>2b)b}{a^5}\left|\int_{1/a}^\infty dq \partial_q^2[q\beta(q\leq M) f(q)]\sin(aq)e^{-ibq} \right|\\
\lesssim_M \left(\frac{\chi_1(a>2b)}{a^4}+\frac{\chi_1(a>2b)}{a^4b}\right)C(f).\label{>11}
\end{multline}
According to \eqref{>11}, \eqref{>12}, \eqref{>2}, we have
\eq
|I_{>}(a,b)|\lesssim_M \left(\frac{1}{\langle a+b\rangle^2\langle a-b\rangle^2}+\frac{1}{ \langle a+b\rangle^2\langle a-b\rangle^2b}\right)C(f).\label{>}
\eeq
\textbf{Estimate for $I_{<}(a,b)$:} take integration by parts in $q$ variable by setting 
\eq
e^{-ibq}=\frac{1}{-ib}\partial_{q}[e^{-ibq}]
\eeq
and we have 
\begin{multline}
I_<(a,b)=\frac{\chi_1(a<b/2)}{ab}\times \frac{1}{-ib}\left( q\beta(q\leq M)f(q)\sin(aq)e^{-ibq}\vert_{q=0}^{q=\infty}-\right.\\
\int_0^\infty dq \partial_q[q\beta(q\leq M)f(q) ][\sin(aq)e^{-ibq}]-\left.\int_0^\infty dq q\beta(q\leq M) f(q)[a\cos(aq)e^{-ibq}]\right)\\
=\frac{\chi_1(a<b/2)}{-iab^2}\int_0^\infty dq \partial_q[q\beta(q\leq M)f(q) ]\sin(aq)e^{-ibq}+\frac{i\chi_1(a<b/2)}{-ib^2}\int_0^\infty dq q\beta(q\leq M)f(q) \cos(aq)e^{-ibq}\\
=:I_{<,1}(a,b)+I_{<,2}(a,b).
\end{multline}
Here we discard the boundary terms due to factor $q$ and factor $\beta(q\leq M)$. 

For $I_{<,2}(a,b)$, we do the same transformation as what we did for $I_\sim(a,b)$ and similarly, we have
\begin{multline}
|I_{<,2}(a,b)|\lesssim \frac{\chi_1(a<b/2)}{b^2 \langle a-b\rangle^2}\times \left(|f(0)| +\sup\limits_{q\in [0,2M]}\| \partial_q^2[q\beta(q\leq M)f(q)]\|_{\s^1_q}\right)\\
(\text{due to }\chi_1(a<b/2))\lesssim_M   \frac{1}{\langle a+b\rangle^2 \langle a-b\rangle^2}C(f). \label{<2}
\end{multline}
For $I_{<,1}(a,b)$, we take integration by parts in the same way again and have
\begin{multline}
I_{<,1}(a,b)=\frac{\chi_1(a<b/2)}{(-i)^2ab^3}\left(\partial_q[q\beta(q\leq M)f(q) ]\sin(aq)e^{-ibq}\vert_{q=0}^{q=\infty}- \right.\\
\left.\int_0^\infty dq \partial_q^2[q\beta(q\leq M)f(q) ]\sin(aq)e^{-ibq}-\int_0^\infty dq\partial_q[q\beta(q\leq M)f(q) ](a\cos(aq)e^{-ibq})\right)\\
=\frac{\chi_1(a<b/2)}{ab^3}\int_0^\infty dq \partial_q^2[q\beta(q\leq M)f(q) ]\sin(aq)e^{-ibq}+\\
\frac{\chi_1(a<b/2)}{b^3}\int_0^\infty dq\partial_q[q\beta(q\leq M)f(q) ]\cos(aq)e^{-ibq}\\
=:I_{<,11}(a,b)+I_{<,12}(a,b).
\end{multline}
Here we discard the boundary terms due to factor $\sin(aq)$ and factor $\beta(q\leq M)$. 

For $I_{<,12}(a,b)$, we take the same transformation as what we did for $I_{>,12}(a,b)$ and we have
\begin{multline}
|I_{<,12}(a,b)|\lesssim \frac{\chi_1(a<b/2)}{b^3 \langle a-b\rangle}\times \left(|f(0)| +\sup\limits_{q\in [0,2M]}\| \partial_q^2[q\beta(q\leq M)f(q)]\|_{\s^1_q}\right)\\
(\text{due to }\chi_1(a<b/2))\lesssim_M   \frac{1}{\langle a+b\rangle^2 \langle a-b\rangle^2}C(f). \label{<12}
\end{multline}
For $I_{<,11}(a,b)$, we take integration by parts in the same way again and have
\begin{multline}
|I_{<,11}(a,b)|\lesssim \frac{\chi_1(a<b/2)}{ab^4}|  \partial_q^2[q\beta(q\leq M)f(q) ]\sin(aq)e^{-ibq}|\vert_{q=0}^{q=\infty}+\\
\frac{\chi_1(a<b/2)}{b^4}|\int_0^\infty dq \partial_q^3[q\beta(q\leq M)f(q) ]\frac{\sin(aq)}{a}e^{-ibq}|+\\
\frac{\chi_1(a<b/2)}{b^4}|\int_0^\infty dq \partial_q^2[q\beta(q\leq M)f(q) ]\cos(aq)e^{-ibq}|\\
(\text{due to }\chi_1(a<b/2))\lesssim_M \frac{1}{\langle a+b\rangle^2 \langle a-b\rangle^2}C(f). \label{<11}
\end{multline}
Here we discard the boundary terms due to factor $\sin(aq)$ and factor $\beta(q\leq M)$ and we use 
\eq
|\frac{\sin(aq)}{a}|\lesssim q.
\eeq
According to \eqref{<11}, \eqref{<12}, \eqref{<2}, we have 
\eq
|I_{<}(a,b)|\lesssim_M \frac{1}{\langle a+b\rangle^2 \langle a-b\rangle^2}C(f).\label{<}
\eeq
For $I_0(a,b)$, we do nothing and clearly have
\eq
|I_0(a,b)|\lesssim_M \frac{\chi(a+b<10)}{ab}C(f).\label{I0}
\eeq
According to \eqref{sim}, \eqref{<}, \eqref{>} and \eqref{I0}, we get \eqref{I(ab)} and finish the proof.
\end{proof}
\subsection{Proof of Lemma \ref{LpI11} and Lemma \ref{LpI21}}
\begin{lemma}\label{Osi}If $R(z) $ satisfies \eqref{RVd} and \eqref{RV0}, and if $V(x)\in \mathcal{K}_{\delta,x}^1\cap\s^1_{\delta,x}\cap \s^\infty_{2\delta,x}$, then for any $y\in \mathbb{R}^3$,
\begin{align}
\left|\int d^3k\right.&\left.\int_0^\infty qdq\beta(q\leq M) \frac{e^{-i|k|q}}{|k|}V(x-k) [R_1^-(q^2)V(x)\frac{\sin(|x-y|q) }{|x-y|}](x-k)\right|\\
\lesssim_M&\int d^3z d^3k|V(x-k)|\langle z\rangle^\delta|V(z)|\frac{\langle k\rangle}{|k|} \frac{\chi(|k|+|z-y|>1)}{(|k|+|z-y|)^2} \frac{1}{\langle|k|-|z-y|\rangle^2}+\\
&\int d^3z d^3k|V(x-k)|\langle z\rangle^\delta|V(z)| \frac{\chi(|k|+|z-y|\leq 10)}{|k||z-y|}.\label{maineq1}
\end{align}
In particular, 
\begin{align}
&\left|\int_0^\infty qdq\beta(q\leq M) \frac{e^{-i|k|q}}{|k|}V(x-k)\frac{\sin(|x-k-y|q) }{|x-k-y|}\right|\\
=&\left|\int d^3k\int_0^\infty qdq\beta(q\leq M) \frac{e^{-i|k|q}}{|k|}V(x-k)\frac{\sin(|x-k-y|q) }{|x-k-y|}\right|\\
\lesssim_M&\int  d^3k |V(x-k)|\frac{\langle k\rangle}{|k|}\frac{\chi(|k|+|x-k-y|>1)}{(|k|+|x-k-y|)^2} \frac{1}{\langle|k|-|x-k-y|\rangle^2}+\\
&\int  d^3k |V(x-k)|\frac{\chi(|k|+|x-k-y|\leq10)}{|k||x-k-y|}\label{maineq2}
\end{align}
and 
\begin{align}
&\left|\int_0^\infty qdq\beta(q\leq M) \frac{e^{-i|k|q}}{|k|}V(x-k)[R_0^-(q^2)P_cV(x)\frac{\sin(|x-y|q) }{|x-y|}](x-k)\right|\\
=&c\left|\int d^3kd^3p\int_0^\infty qdq\beta(q\leq M) \frac{e^{-i|k|q}}{|k|}V(x-k)\frac{e^{-i|p|q}}{|p|}V(x-k-p)\frac{\sin(|x-k-p-y|q) }{|x-k-p-y|}\right|\\
\lesssim_M&\int  d^3kd^3p |V(x-k)|\frac{|V(x-k-p)|}{|p||k|}\frac{\chi(|k|+|x-k-p-y|>1)}{(|k|+|x-k-p-y|)^2} \frac{\langle k\rangle}{\langle|k|-|x-k-p-y|\rangle^2}+\\
&\int  d^3kd^3p |V(x-k)| \frac{|V(x-k-p)|}{|p|}\frac{\chi(|k|+|x-k-p-y|\leq10)}{|k||x-k-p-y|}\label{maineq8}
\end{align}
\end{lemma}
\begin{proof}
Setting 
\eq
a=|x-k-z-y|, b=|k|
\eeq
\eq
f(q)\sin(aq)=\int d^3z V(x-k)R^-_1(q^2; z) V(x-k-z)\frac{\sin(|x-k-z-y|q)}{|x-k-z-y|}
\eeq
using 
\eq
1+\frac{1}{|k|}\sim \frac{\langle k\rangle}{|k|}
\eeq
and applying Lemma \ref{main}, Lemma \ref{resolvent}, changing variables from $x-k-z$ to $h=x-k-z$, we have 
\begin{align}
\left|\int d^3k\right.&\left.\int_0^\infty qdq\beta(q\leq M) \frac{e^{-i|k|q}}{|k|}V(x-k) [R_1^-(q^2)V(x)\frac{\sin(|x-y|q) }{|x-y|}](x-k)\right|\\
\lesssim_M&\int d^3h d^3k|V(x-k)|\langle h\rangle^\delta |V(h)|\frac{\langle k\rangle}{|k|} \frac{\chi(|k|+|h-y|>1)}{(|k|+|h-y|)^2} \frac{1}{\langle|k|-|h-y|\rangle^2}+\\
&\int d^3h d^3k |V(x-k)|\langle h\rangle^\delta |V(h)| \frac{\chi(|k|+|h-y|\leq 10)}{|k||h-y|}.
\end{align}
Similarly, we have \eqref{maineq2} and \eqref{maineq8}.

\end{proof}

\begin{corollary}\label{Osi2}If $R(z) $ satisfies \eqref{RVd} and \eqref{RV0}, and if $V(x)\in \s^1_{\delta+1,x}\cap \s^\infty_{2\delta,x}$, then for any $y\in \mathbb{R}^3$,
\begin{align}
\left|\int d^3k\right.&\int_0^\infty qdq\beta(q\leq M) \frac{e^{-i|k|q}}{|k|}V(x-k) [R_1(q^2)V(x)\frac{\sin|x-y|q }{|x-y|}](x-k)|y|-\\
\int d^3k&\left.\int_0^\infty qdq\beta(q\leq M) \frac{e^{-i|k|q}}{|k|}V(x-k) [R_1(q^2)V(x)\frac{\sin|x-y|q }{|x-y|}](x-k)|x|\right|\\
\lesssim_M&\int d^3z d^3k\langle x-k\rangle|V(x-k)|\langle z\rangle^{\delta+1}|V(z)|\frac{\langle k\rangle}{|k|} \frac{\chi(|k|+|z-y|>1)}{(|k|+|z-y|)^2} \frac{1}{\langle|k|-|z-y|\rangle}+\label{cm1}\\
&\int d^3z d^3k\langle x-k\rangle|V(x-k)|\langle z\rangle^{\delta+1}|V(z)| \frac{\chi(|k|+|z-y|\leq 10)}{|k||z-y|}.\label{cm2}
\end{align}
In particular, 
\begin{align}
&\left|\int d^3k\int_0^\infty qdq\beta(q\leq M) \frac{e^{-i|k|q}}{|k|}V(x-k)\frac{\sin|x-k-y|q }{|x-k-y|}|y|-\right.\\
&\left.  \int d^3k\int_0^\infty qdq\beta(q\leq M) \frac{e^{-i|k|q}}{|k|}V(x-k)\frac{\sin|x-k-y|q }{|x-k-y|}|x|\right|\\
\lesssim_M&\int  d^3k \langle x-k\rangle|V(x-k)|\frac{\langle k\rangle}{|k|}\frac{\chi(|k|+|x-k-y|>1)}{(|k|+|x-k-y|)^2} \frac{1}{\langle|k|-|x-k-y|\rangle}+\label{cm3}\\
&\int  d^3k \langle x-k\rangle|V(x-k)|\frac{\chi(|k|+|x-k-y|\leq10)||x|-|y||}{|k||x-k-y|}\label{cm4}
\end{align}
and
\begin{multline}
\left|\int d^3k\int_0^\infty qdq\beta(q\leq M) \frac{e^{-i|k|q}}{|k|}V(x-k)[R_0^-(q^2)P_cV(x)\frac{\sin|x-y|q }{|x-y|}|y|](x-k)-\right.\\
\left.  \int d^3k\int_0^\infty qdq\beta(q\leq M) \frac{e^{-i|k|q}}{|k|}V(x-k)[R_0^-(q^2)P_cV(x)\frac{\sin|x-y|q }{|x-y|}](x-k)|x|\right|\\
\lesssim_M\int  d^3k d^3q\frac{\langle x-k-q\rangle |V(x-k-q)|}{|q|}\frac{\chi(|k|+|x-k-q-y|>1)}{(|k|+|x-k-q-y|)^2} \frac{\langle x-k\rangle|V(x-k)|}{\langle|k|-|x-k-q-y|\rangle}\frac{\langle k\rangle}{|k|}+\\
\int  d^3k d^3q\langle x-k\rangle|V(x-k)| \frac{\langle x-k-q\rangle}{|q|}|V(x-k-q)| \frac{\chi(|k|+|x-k-q-y|\leq10)||k|-|x-k-q-y||}{|k||x-k-q-y|}.\label{cm8}
\end{multline}
\end{corollary}
\begin{proof}
Since
\eq
|x|-|y|\leq (|x-k|+|k|)-(|z-y|-|z|)\leq ||k|-|z-y||+|k|+|z|, 
\eeq
and
\eq
|x|-|y|\geq (|k|-|x-k| )-(|z-y|+|z|)\geq -|| k|-|z-y||-|x-k|-|z|,
\eeq
we get \eqref{cm1} and \eqref{cm2} by using the same argument we did for Lemma \ref{Osi}. Similarly, we get \eqref{cm3}, \eqref{cm4} and \eqref{cm8}.
\end{proof}
Let
\eq
F_1(x,y):=\int d^3k\int_0^\infty qdq\beta(q\leq M) \frac{e^{-i|k|q}}{|k|}V(x-k) [R_1^-(q^2)V(x)\frac{\sin(|x-y|q) }{|x-y|}](x-k),
\eeq
\eq
F_2(x,y):=-\int d^3k\int_0^\infty qdq\beta(q\leq M) \frac{e^{-i|k|q}}{|k|}V(x-k) [R_0^-(q^2)P_cV(x)\frac{\sin(|x-y|q) }{|x-y|}](x-k)
\eeq
and 
\eq
F(x,y):=\int d^3k\int_0^\infty qdq\beta(q\leq M) \frac{e^{-i|k|q}}{|k|}V(x-k) [R^-(q^2)P_cV(x)\frac{\sin(|x-y|q) }{|x-y|}](x-k).
\eeq
Then 
\eq
F(x,y)=F_1(x,y)+F_2(x,y)
\eeq
and 
\eq
I_2=c\int d^3y F(x,y)\psi(y)
\eeq
for some constant $c$, see \eqref{I1I2}.
\begin{proof}[\bf Proof of Lemma \ref{LpI21}]For $\psi\in \s^\infty_x$, since
\begin{align}
&\int d^3y \frac{\chi(|k|+|z-y|>1)}{(|k|+|z-y|)^2}\frac{1}{\langle|k|-|z-y| \rangle^2}\\
=& \int d^3\xi \frac{\chi(|k|+|\xi|>1)}{(|k|+|\xi|)^2}\frac{1}{\langle|k|-|\xi| \rangle^2}\\
\lesssim& 1,
\end{align}
and 
\eq
\int d^3y \frac{\chi(|k|+|z-y|\leq 10)}{|z-y|}\lesssim 1,
\eeq
according to Lemma \ref{Osi}, we have
\begin{multline}
\|\int d^3y F_1(x,y)\psi(y)\|_{\s^\infty_x}\lesssim_M\\
\int d^3z d^3k|V(x-k)|\langle z\rangle^\delta|V(z)|\frac{\langle k\rangle}{|k|} \|\frac{\chi(|k|+|z-y|>1)}{(|k|+|z-y|)^2} \frac{1}{\langle|k|-|z-y|\rangle^2}\|_{\s^1_y}\|\psi(x)\|_{\s^\infty_x}+\\
\int d^3z d^3k|V(x-k)|\langle z\rangle^\delta|V(z)| \|\frac{\chi(|k|+|z-y|\leq 10)}{|k||z-y|}\|_{\s^1_y}\|\psi(x)\|_{\s^\infty_x}\\
\lesssim_M(  \| V(x)\|_{\mathcal{K}^1_{\delta,x}}+\| V(x)\|_{\s^1_{\delta,x}}) \| V(x)\|_{\s^1_{\delta,x}} \|\psi(x)\|_{\s^\infty_x}.\label{Finfty1}
\end{multline}
Similarly, we have
\begin{multline}
\|\int d^3y F_2(x,y)\psi(y)\|_{\s^\infty_x}\lesssim_M\int d^3k d^3p \frac{\langle k\rangle}{|k|}|V(x-k)|\frac{|V(x-k-p)|}{|p|}\|\psi(x)\|_{\s^\infty_x}\\
\lesssim \| V(x)\|_{\mathcal{K}^1_{\delta,x}}(\| V(x)\|_{\s^1_{\delta,x}}+ \| V(x)\|_{\mathcal{K}^1_{\delta,x}})\|\psi(x)\|_{\s^\infty_x}.\label{Finfty2}
\end{multline}
\eqref{Finfty1} and \eqref{Finfty2} imply that 
\eq
\| I_2\psi(x)\|_{\s^\infty_x}\lesssim \|\int d^3y F(x,y)\psi(y)\|_{\s^\infty_x}\lesssim_M \|\psi(x)\|_{\s^\infty_x}.
\eeq
For $\psi\in \s^1_x$, since
\begin{align}
\int d^3k \frac{\langle k\rangle}{|k|}\frac{\chi(|k|+|z-y|>1)}{(|k|+|z-y|)^2}\frac{1}{\langle|k|-|z-y| \rangle^2}\lesssim& 1,
\end{align}
and 
\eq
\int d^3k \frac{\chi(|k|+|z-y|\leq 10)}{|k|}\lesssim 1,
\eeq
we have
\eq
\|\int d^3y F_1(x,y)\psi(y)\|_{\s^1_x}\lesssim_M(  \| V(x)\|_{\mathcal{K}^1_{\delta,x}}+\| V(x)\|_{\s^1_{\delta,x}}) \| V(x)\|_{\s^1_{\delta,x}} \|\psi(x)\|_{\s^1_x}.\label{F11}
\eeq
Similarly, we have
\begin{multline}
\|\int d^3y F_2(x,y)\psi(y)\|_{\s^1_x}\lesssim_M\\
\int d^3k d^3p |V(x-k)|\frac{|V(x-k-p)|}{|p|}\|\psi(x)\|_{\s^\infty_x}\lesssim \| V(x)\|_{\mathcal{K}^1_{\delta,x}}\| V(x)\|_{\s^1_{\delta,x}}\|\psi(x)\|_{\s^1_x}.\label{F12}
\end{multline}
\eqref{F11} and \eqref{F12} imply 
\eq
\| I_2\psi(x)\|_{\s^1_x}\lesssim\| \int d^3y F(x,y)\psi(y)\|_{\s^1_x}\lesssim_M \|\psi(x)\|_{\s^1_x}.
\eeq
By interpolation inequality, we finish the proof.
\end{proof}
\begin{proof}[\bf Proof of Lemma \ref{LpI11}]It follows by using the same process as what we did for Lemma \ref{LpI21} except that we use \eqref{maineq2} instead of \eqref{maineq1} and \eqref{maineq8}.
\end{proof}
\begin{proof}[Proof of Lemma \ref{ReWave}]
It follows from
\begin{multline}
(\Omega_+-P_c)\psi(x)=i\int_0^\infty dt e^{-0t}P_ce^{itH}V(x)e^{-itH_0}\psi(x)\\
=\frac{i}{(2\pi)^3}\int d^3qd^3yP_cR^-(q^2)V(x)e^{i(x-y)\cdot q} \psi(y).\label{Sep-1}
\end{multline}
\end{proof}

\begin{proof}[Proof of Lemma \ref{sphere}]
\begin{align}
\int_{S^2}d\sigma(q)e^{ix\cdot q}=& \int_0^{2\pi} d\phi \int_{-\frac{\pi}{2}}^{\frac{\pi}{2}} \sin\theta d\theta e^{i|x||q|\cos \theta}\\
=& 2\pi\int_{-1}^1 dt e^{it|x||q|}\\
=&\frac{4\pi i \sin(|x||q|)}{|x||q|}.
\end{align}
\end{proof}
\subsection{Proof of Theorem \ref{thm1}}
Now it is time to prove Theorem \ref{thm1}.
\begin{proof}[Proof of Theorem \ref{thm1}]According to Lemma \ref{ReWave}, it is sufficient to estimate 
\eq
I_1+I_2=i\int_0^\infty dt P_ce^{itH}V(x)e^{-itH_0}\beta(|P|\leq M)
\eeq
Then Theorem \ref{thm1} follows from Lemma \ref{LpI11} and Lemma \ref{LpI21}.
\end{proof}

\section{Application to $[|x|, \Omega_+]$}
Now we prove Lemma \ref{com}.
\begin{proof}[Proof of Lemma \ref{com}]When $p=1$, choose $\psi\in \s^1_x$. According to Corollary \ref{ReWave}, we get
\begin{multline}
[|x|,(\Omega_+-P_c)\beta(|P|\leq M)]\psi(x)= |x|\frac{1}{(2\pi)^3}\int d^3q d^3yP_cR^-(q^2)V(x)e^{i(x-y)\cdot q} \beta(|q|\leq M)\psi(y)-\\
\frac{1}{(2\pi)^3}\int d^3q d^3yP_cR^-(q^2)V(x)e^{i(x-y)\cdot q} \beta(|q|\leq M)|y|\psi(y).
\end{multline}
Let
\begin{multline}
CF_1(x,y):=|x|\frac{1}{(2\pi)^3}\int d^3q R_0^-(q^2)V(x)R_1(q^2)V(x)e^{i(x-y)\cdot q} \beta(|q|\leq M)\psi(y)-\\
\frac{1}{(2\pi)^3}\int d^3q d^3yR_0^-(q^2)V(x)R_1(q^2)V(x)e^{i(x-y)\cdot q} \beta(|q|\leq M)|y|,
\end{multline}
\begin{multline}
CF_2(x,y):=|x|\frac{1}{(2\pi)^3}\int d^3qR_0^-(q^2)V(x)R_0(q^2)P_cV(x)e^{i(x-y)\cdot q} \beta(|q|\leq M)\psi(y)-\\
\frac{1}{(2\pi)^3}\int d^3q d^3yR_0^-(q^2)V(x)R_0(q^2)P_cV(x)e^{i(x-y)\cdot q} \beta(|q|\leq M)|y|
\end{multline}
and 
\begin{multline}
CF(x,y):=|x|\frac{1}{(2\pi)^3}\int d^3qR_0^-(q^2)V(x)R^-(q^2)P_cV(x)e^{i(x-y)\cdot q} \beta(|q|\leq M)\psi(y)-\\
\frac{1}{(2\pi)^3}\int d^3q d^3yR_0^-(q^2)V(x)R^-(q^2)P_cV(x)e^{i(x-y)\cdot q} \beta(|q|\leq M)|y|.
\end{multline}
Then 
\eq
CF(x,y)=CF_1(x,y)+CF_2(x,y).
\eeq
According to \eqref{cm1}, \eqref{cm2}, we have
\begin{multline}
|CF_1(x,y)|\lesssim_M\int d^3z d^3k\langle x-k\rangle|V(x-k)|\langle z\rangle^{\delta+1}|V(z)|\frac{\langle k\rangle}{|k|} \frac{\chi(|k|+|z-y|>1)}{(|k|+|z-y|)^2} \frac{1}{\langle|k|-|z-y|\rangle}+\\
\int d^3z d^3k\langle x-k\rangle|V(x-k)|\langle z\rangle^{\delta+1}|V(z)| \frac{\chi(|k|+|z-y|\leq 10)}{|k||z-y|}.
\end{multline}
Changing variables from $k$ to $\eta=x-k$, we get
\begin{multline}
|CF_1(x,y)|\lesssim_M\int d^3z d^3\eta\langle \eta\rangle|V(\eta)|\langle z\rangle^{\delta+1}|V(z)|\frac{\langle \eta-x\rangle}{|\eta-x|} \frac{\chi(|x-\eta|+|z-y|>1)}{(|x-\eta|+|z-y|)^2} \frac{1}{\langle|x-\eta|-|z-y|\rangle}+\\
\int d^3z d^3\eta\langle \eta\rangle|V(\eta)|\langle z\rangle^{\delta+1}|V(z)| \frac{\chi(|x-\eta|+|z-y|\leq 10)}{|x-\eta||z-y|}.
\end{multline}
Since
\eq
\|\frac{\langle \eta-x\rangle}{|\eta-x|} \frac{\chi(|x-\eta|+|z-y|>1)}{(|x-\eta|+|z-y|)^2} \frac{1}{\langle|x-\eta|-|z-y|\rangle}\|_{\s^1_x+\s^{1+\epsilon}_x}\lesssim 1
\eeq
for any $\epsilon>0$ and since
\eq
\|\frac{\chi(|x-\eta|+|z-y|\leq 10)}{|x-\eta||z-y|}\|_{\s^1_x}\lesssim 1,
\eeq
we have 
\eq
\| CF_1(x,y)\|_{\s^1_x+\s^{1+\epsilon}_x}\lesssim_M \|V(x)\|_{\s^1_{\delta+1,x}}\|V(x)\|_{\s^1_{1,x}}.
\eeq
With $\psi(y)\in \s^1_y$, we have
\begin{multline}
\|\int d^3y CF_1(x,y)\psi(y)\|_{\s^1_x+\s^{1+\epsilon}_x}\lesssim \int d^3y |\psi(y)| \| CF_1(x,y)\|_{\s^1_x+\s^{1+\epsilon}_x}\\
\lesssim_M \|\psi(x)\|_{\s^1_x}.\label{CF1}
\end{multline}
Similarly, we have
\eq
\|\int d^3y CF_2(x,y)\psi(y)\|_{\s^1_x+\s^{1+\epsilon}_x}\lesssim_M \|V(x)\|_{\mathcal{K}^1_{1,x}}\|V(x)\|_{\s^1_{1,x}} \|\psi(x)\|_{\s^1_x}.\label{CF2}
\eeq
Thus, \eqref{CF1}, \eqref{CF2} imply
\eq
\|\int d^3y CF(x,y)\psi(y)\|_{\s^1_x+\s^{1+\epsilon}_x}\lesssim_M \|\psi(x)\|_{\s^1_x}.\label{CF}
\eeq
Let
\begin{multline}
CF_0(x,y):=|x|\frac{1}{(2\pi)^3}\int d^3qR_0^-(q^2)P_cV(x)e^{i(x-y)\cdot q} \beta(|q|\leq M)\psi(y)-\\
\frac{1}{(2\pi)^3}\int d^3q d^3yR_0^-(q^2)P_cV(x)e^{i(x-y)\cdot q} \beta(|q|\leq M)|y|.
\end{multline}
By using \eqref{cm3}, \eqref{cm4} instead of \eqref{cm1}, \eqref{cm2}, similarly, we have
\eq
\|\int d^3y CF_0(x,y)\psi(y)\|_{\s^1_x+\s^{1+\epsilon}_x}\lesssim_M\|V(x)\|_{\s^1_{1,x}} \|\psi(x)\|_{\s^1_x}.\label{CF0}
\eeq
Since
\eq
[|x|,(\Omega_+-P_c)\beta(|P|\leq M)]\psi(x)=\int d^3y CF_0(x,y)\psi(y)-\int d^3y CF(x,y)\psi(y),
\eeq
we have
\eq
\|[|x|,(\Omega_+-P_c)\beta(|P|\leq M)]\psi(x) \|_{\s^{1}_x+\s^{1+\epsilon}_x}\lesssim_M \|\psi(x)\|_{\s^1_x}.
\eeq
When $p\in (1,\infty)$, Since
\eq
\| \frac{\chi(|x-\eta|+|z-y|>1)}{(|x-\eta|+|z-y|)^2} \frac{1}{\langle|x-\eta|-|z-y|\rangle}\|_{\s^{1+\epsilon}_y}\lesssim 1
\eeq
for any $\epsilon>0$ and since
\eq
\|\frac{\chi(|x-\eta|+|z-y|\leq 10)}{|x-\eta||z-y|}\|_{\s^{1+\epsilon}_y}\lesssim 1,
\eeq
by using H\"older's inequality, 
\eq
\| \frac{\chi(|x-\eta|+|z-y|>1)}{(|x-\eta|+|z-y|)^2} \frac{\psi(y)}{\langle|x-\eta|-|z-y|\rangle}\|_{\s^1_y}\lesssim_p \|\psi(x)\|_{\s^p_x}
\eeq
and 
\eq
\|\psi(y)\frac{\chi(|x-\eta|+|z-y|\leq 10)}{|x-\eta||z-y|}\|_{\s^{1}_y}\lesssim_p \|\psi(x)\|_{\s^p_x}
\eeq
which implies 
\eq
\|\int d^3y CF_1(x,y)\psi(y)\|_{\s^\infty}\lesssim_{M,p} \|V(x)\|_{\s^1_{\delta+1,x}}(\|V(x)\|_{\s^1_{1,x}}+\|V(x)\|_{\mathcal{K}^1_{1,x}})\|\psi(x)\|_{\s^p_x}.\label{CF10}
\eeq
Similarly, 
\eq
\|\int d^3y CF_2(x,y)\psi(y)\|_{\s^\infty_x}\lesssim_{M,p} \|V(x)\|_{\mathcal{K}^1_{1,x}}(\|V(x)\|_{\s^1_{1,x}}+\|V(x)\|_{\mathcal{K}^1_{1,x}}) \|\psi(x)\|_{\s^p_x}.\label{CF20}
\eeq
and 
\eq
\|\int d^3y CF_0(x,y)\psi(y)\|_{\s^\infty_x}\lesssim_{M,p}(\|V(x)\|_{\s^1_{1,x}}+\|V(x)\|_{\mathcal{K}^1_{1,x}}) \|\psi(x)\|_{\s^p_x}.\label{CF00}
\eeq
Thus, we have
\eq
\|[|x|,(\Omega_+-P_c)\beta(|P|\leq M)]\psi(x) \|_{\s^\infty_x}\lesssim_{M,p} \|\psi(x)\|_{\s^p_x}.
\eeq
By interpolation inequality, we get desired result and finish the proof.

\end{proof}

\bibliographystyle{uncrt}

\begin{thebibliography}{99}
\bibitem{B2011}
Beceanu, M. (2011). New estimates for a time-dependent Schr\"odinger equation. Duke Mathematical Journal, 159(3), 417-477.

\bibitem{GJY2004}
Galtbayar, A., Jensen, A., \& Yajima, K. (2004). Local time-decay of solutions to Schr\"odinger equations with time-periodic potentials. Journal of Statistical Physics, 116(1), 231-282.

\bibitem{G2004}
Goldberg, M. (2004). Dispersive bounds for the three-dimensional Schr\"odinger equation with almost critical potentials. arXiv preprint math/0409327.

\bibitem{GS2004}
Goldberg, M., \& Schlag, W. (2004). A limiting absorption principle for the three-dimensional Schr\"odinger equation with L p potentials. International Mathematics Research Notices, 2004(75), 4049-4071.

\bibitem{G2006}
Goldberg, M. (2006). Dispersive estimates for the three-dimensional Schr\"odinge equation with rough potentials. American journal of mathematics, 128(3), 731-750.

\bibitem{JK1979}
Jensen, A., \& Kato, T. (1979). Spectral properties of Schr\"odinger operators and time-decay of the wave functions. Duke mathematical journal, 46(3), 583-611.

\bibitem{JSS1990}
Journ\'e, J. L., Soffer, A., \& Sogge, C. D. (1990). $L^ p\to L^{p'} $ estimates for time-dependent Schr\"odinger operators. Bulletin (New Series) of the American Mathematical Society, 23(2), 519-524.

\bibitem{KK2014}
Komech, A., \& Kopylova, E. (2014). Dispersion decay and scattering theory. John Wiley \& Sons.

\bibitem{RS2004}
Rodnianski, I., \& Schlag, W. (2004). Time decay for solutions of Schr\"odinger equations with rough and time-dependent potentials. Inventiones mathematicae, 155(3), 451-513.

\bibitem{S2007}
Schlag, W. (2007). Dispersive estimates for Schr\"odinger operators: a survey. Mathematical aspects of nonlinear dispersive equations, 163, 255-285.

\bibitem{SW2020}
Soffer, A., \& Wu, X. (2020). $L^p$ Boundedness of the Scattering Wave Operators of Schr\"oedinger Dynamics with Time-dependent Potentials and applications. arXiv preprint arXiv:2012.14356.

\bibitem{Y1995}
Yajima, K. (1995). The $W^{k,p}$-continuity of wave operators for Schr\"odinger operators. Journal of The Mathematical Society of Japan, 47, 551-581.





\end{thebibliography}
\def\bibfont{\HUGE}

\end{document}